\def\timestamp{%
Time-stamp: <L-reflection.tex: Wednesday 17-07-2013 at 12:37:45 (cest)>}
\def\stripname Time-stamp: <#1 #2>{#2}
\edef\filedate{\expandafter\stripname\timestamp}

\documentclass{amsart}

\DeclareMathSymbol\N    0{AMSb}{`N}
\DeclareMathSymbol\Poset0{AMSb}{`P}
\DeclareMathSymbol\Qoset0{AMSb}{`Q}  

\DeclareMathSymbol\restr\mathbin{AMSa}{"16}  
\let\models\undefined
\DeclareMathSymbol\models\mathrel{AMSa}{"0F}
\DeclareMathSymbol\forces\mathrel{AMSa}{"0D}
\DeclareMathSymbol\le    \mathrel{AMSa}{"36}
\DeclareMathSymbol\ge    \mathrel{AMSa}{"3E}
\newcommand\calA{\mathcal{A}}
\newcommand\calB{\mathcal{B}}

\newcommand\calF{\mathcal{F}}
\newcommand\calU{\mathcal{U}}

\newcommand\axiom{\mathsf}
\newcommand\CH{\axiom{CH}}

\newcommand\orpr[2]{\langle#1,#2\rangle}
\newcommand\dom{\operatorname{dom}}
\newcommand\val{\operatorname{val}}

\newcommand\cl{\operatorname{cl}}

\newcommand\bee{\mathfrak{b}}
\newcommand\cee{\mathfrak{c}}

\newcommand\functions[2]{{}^{#1}#2}

\let\fname\dot

\newcommand\preim{^\gets}

\newcommand\omegaseq[1]{\langle{#1}_n:n\in\omega\rangle}
\newcommand\omegaoneseq[1]{\langle{#1}_\alpha:\alpha\in\omega_1\rangle}

\theoremstyle{plain}
\newtheorem{theorem}{Theorem}[section]
\newtheorem{lemma}[theorem]{Lemma}
\newtheorem{proposition}[theorem]{Proposition}
\newtheorem{corollary}[theorem]{Corollary}

\theoremstyle{definition}
\newtheorem{definition}[theorem]{Definition}

\theoremstyle{remark}

\newtheorem{question}{Question}
\newtheorem{example}[theorem]{Example}
\numberwithin{equation}{section}

\usepackage{amsrefs}

\begin{document}

\title{Reflecting Lindel\"of and converging $\omega_1$-sequences}

\author[Alan Dow]{Alan Dow\dag}
\address{
Department of Mathematics\\
UNC-Charlotte\\
9201 University City Blvd. \\
Charlotte, NC 28223-0001}
\email{adow@uncc.edu}
\urladdr{http://www.math.uncc.edu/\~{}adow}
\thanks{\dag Research of the first author was supported by NSF grant 
             No.\ NSF-DMS-0901168.} 

\author{Klaas Pieter Hart}
\address{Faculty of Electrical Engineering, Mathematics and Computer Science\\  
         TU Delft\\
         Postbus 5031\\
         2600~GA {} Delft\\
         the Netherlands}
\email{k.p.hart@tudelft.nl}
\urladdr{http://fa.its.tudelft.nl/\~{}hart}

\keywords{compact space, first-countable space, Lindel\"ofness, 
          converging sequence, L-reflection, small diagonal, forcing,
          property K}
\subjclass{Primary: 54D30. 
           Secondary: 03E35, 54A20, 54A25, 54A35, 54D20}


\date{\filedate}

\begin{abstract}
We deal with a conjectured dichotomy for compact Hausdorff spaces:
each such space contains a non-trivial converging $\omega$-sequence
or a non-trivial converging $\omega_1$-sequence.
We establish that this dichotomy holds in a variety of models; these include
the Cohen models, the random real models and any model obtained from a model
of~$\CH$ by an iteration of property~$K$ posets.
In fact in these models every compact Hausdorff space without non-trivial 
converging $\omega_1$-sequences is first-countable and, in addition,
has many $\aleph_1$-sized Lindel\"of subspaces.
As a corollary we find that in these models all compact Hausdorff spaces
with a small diagonal are metrizable.
\end{abstract}

\maketitle

\section*{Introduction}

This paper deals with converging sequences of type~$\omega$ and~$\omega_1$.
If $\gamma$~is a limit ordinal then a sequence 
$\langle x_\alpha:\alpha<\gamma\rangle$ in a topological space is said
to converge to a point~$x$ if for every neighbourhood~$U$ of~$x$
there is an~$\alpha<\gamma$ such that $x_\beta\in U$ for $\beta\ge\alpha$.
To avoid non-relevant cases we shall always assume that our sequences
are injective.

In~\cite{JuSzCvgt} Juh\'asz and Szentmikl\'ossy showed that if a compact space
has a free sequence of length~$\omega_1$ then it has a converging free sequence
of that length --- a sequence $\omegaoneseq{x}$ is \emph{free} if for 
all~$\alpha$ the sets $\{x_\beta:\beta<\alpha\}$ 
and $\{x_\beta:\beta\ge\alpha\}$ have disjoint closures.
One may rephrase this as: a compact space without converging 
$\omega_1$-sequences must have countable tightness.

The authors of~\cite{JuSzCvgt} also recall two questions of Hu\v{s}ek
and Juh\'asz regarding converging $\omega_1$-sequences
\begin{description}
\item[Hu\v{s}ek] does every compact Hausdorff space contain a non-trivial 
     converging $\omega$-sequence or a non-trivial converging 
     $\omega_1$-sequence?
\item[Juh\'asz] does every non first-countable compact Hausdorff space 
     contain a non-trivial converging $\omega_1$-sequence?
\end{description}
In~\cite{MR621987} it was shown that the space~$\beta\N$ does contain
a converging $\omega_1$-sequence, which shows that Hu\v{s}ek's question
is a weakening of Efimov's well-known question in~\cite{MR0253290} 
whether every compact Hausdorff 
space contains a converging $\omega$-sequence or a copy of~$\beta\N$.

For the remainder of the paper we refer to a space without converging 
$\omega_1$-sequence as an \emph{$\omega_1$-free space}. 
Our main result shows that the answer to Juh\'asz' question (and hence
to that of Hu\v{s}ek's) is positive in a large class of models. 
The precise definition will be given later but examples are those
obtained by adding Cohen and random reals and models obtained by iterations
of Hechler forcing.

An important class of $\omega_1$-free spaces consists of those having a 
\emph{small diagonal} --- introduced by Hu\v{s}ek in~\cite{Husek}.
We say that a space, $X$, has a small diagonal if there is no
$\omega_1$-sequence in~$X^2$ that converges to the diagonal~$\Delta(X)$:
a sequence
$\bigl\langle\orpr{x_\alpha}{y_\alpha}:\alpha<\omega_1\bigr\rangle$
in~$X^2$ converges to the diagonal~$\Delta(X)$ if
every neighbourhood of the diagonal contains a tail of the sequence.
Note that if $\omegaoneseq{x}$ converges to~$x$ then 
$\bigl<\orpr{x}{x_\alpha}:\alpha\in\omega\bigr>$ converges to~$\orpr xx$ 
and hence to the diagonal.
A well-studied problem is whether a
compact Hausdorff space with a small diagonal
(which we abbreviate by \emph{csD space}) is metrizable.

The second part of our main result is that, in the same models,
all $\omega_1$-free spaces have many $\aleph_1$-sized Lindel\"of
subspaces.
This will then imply that, in these models, all csD spaces are metrizable. 

We benefit from the results in \cite{JunKosz} concerning the notion of
L-reflecting and the preservation of the Lindel\"of property by
forcing. 

Finally, we should mention that in~\cite{MR1617461} it was shown that
in the Cohen model all csD spaces are metrizable and Juh\'asz' question
has a positive answer.

\subsection*{Status of the problems}

There are consistent examples of compact $\omega_1$-free spaces that are not
first-countable; we review some of these at the end of the paper,
where we will also show that Hu\v{s}ek's question is strictly weaker
than that of~Efimov:
there are many consistent counterexamples to Efimov's question but 
none to Hu\v{s}ek's question (yet).

As to the consistency of the existence of nonmetrizable csD spaces: 
that is still an open problem.

\section{Preliminaries}

\subsection{Elementary sequences and L-reflection}

For a cardinal~$\theta$ we let $H(\theta)$ denote the collection of all sets 
whose transitive closure has cardinality less than~$\theta$ 
(see~\cite{Kunen80}*{Chapter~IV}).   
An $\omega_1$-sequence $\omegaoneseq{M}$ of countable 
elementary substructures of~$H(\theta)$ that satisfies 
$\langle M_\beta:\beta\le\alpha\rangle\in M_{\alpha+1}$ for all~$\alpha$
and $M_\alpha=\bigcup_{\beta<\alpha}M_\beta$ for all limit~$\alpha$
will simply be called an \emph{elementary sequence}. 

Let $X$ be a compact Hausdorff space; an \emph{elementary sequence for~$X$} 
is elementary sequence $\omegaoneseq{M}$ such that $X\in M_0$; of course
in this case $\theta$ should be large enough in order that $X\in H(\theta)$,
in most cases $\theta=(2^\kappa)^+$, where $\kappa=w(X)$, suffices.

\begin{definition} 
We say that a space is weakly L-reflecting if every $\aleph_1$-sized 
subspace is contained in a Lindel\"of subspace of cardinality~$\aleph_1$. 
We say that $X$ is \emph{L-reflecting} if for some regular $\theta$ 
with $X\in H(\theta)$ and any countable~$M_0$ with $X\in M_0\prec H(\theta)$,
there is an elementary sequence $\omegaoneseq{M}$ for~$X$ such that 
$X\cap \bigcup_{\alpha\in \omega_1} M_\alpha$ is  Lindel\"of.
\end{definition}

\subsection{csD spaces}

We will need a technical improvement of Gruenhage's result 
from~\cites{GaryDiag} that hereditarily Lindel\"of csD spaces are metrizable.
It is clear that compact hereditarily Lindel\"of spaces are both first-countable
and L-reflecting; we shall show that the latter two properties suffice
to make csD spaces metrizable.
This will allow us to conclude that in the models from 
Section~\ref{section.forcing}
all csD spaces are metrizable, because they will be seen to
be first-countable and L-reflecting.

We use a convenient characterization of csD spaces obtained by Gruenhage 
in~\cite{GaryDiag}:
for every sequence
$\bigl\langle\orpr{x_\alpha}{y_\alpha}:\alpha\in \omega_1\bigr\rangle$ 
of pairs of points there an uncountable subset~$A$ of~$\omega_1$ such that 
$\{x_\alpha:\alpha\in A\}$ and $\{y_\alpha:\alpha\in A\}$ have disjoint 
closures.

\begin{theorem} 
If a compact first-countable space is L-reflecting and has a small diagonal, 
then it is metrizable.
\end{theorem}

\begin{proof} 
Let $X$ be a csD space that is first-countable and L-reflecting.

Let $\omegaoneseq{M}$ be an elementary sequence for~$X$
such that $X\cap\bigcup_{\alpha<\omega_1}M_\alpha$ is Lindel\"of;
we denote this subspace by~$Y$.
For each $\alpha$ let $\calB_\alpha$ be the family of those open subsets of~$X$
that belong to~$M_\alpha$.

If we assume that $X$ is not metrizable then there is no $\alpha$ for which
$\calB_\alpha$~is a base for the topology of~$X$, or even, by compactness,
a $T_2$-separating open cover of~$X$.
Therefore we can find, by elementarity, points 
$x_\alpha,y_\alpha\in X\cap M_{\alpha+1}$
that do not have disjoint neighbourhoods that belong to~$\calB_\alpha$.

We apply Gruenhage's criterion to find an uncountable
subset~$A$ of~$\omega_1$ such that the closed sets
$F=\cl\{x_\alpha:\alpha\in A\}$ and $G=\cl\{y_\alpha:\alpha\in A\}$
are disjoint. 
We take disjoint open sets $U$ and $V$ around $F$ and $G$ respectively.

Since $X$ is first-countable we know that $M_\alpha$ contains a local base
at each point of $X\cap M_\alpha$.
Thus we may choose for each $x\in Y$ a neighbourhood~$B_x$ such that
\begin{itemize}
\item $B_x\subseteq U$ if $x\in F$
\item $B_x\subseteq V$ if $x\in G$
\item $B_x\cap(F\cup G)=\emptyset$ otherwise
\end{itemize}
and in addition $B_x\in\calB_\alpha$ whenever $x\in M_\alpha$.

Since $Y$~is Lindel\"of there is an $\alpha$ such that 
$Y\subseteq\bigcup\{B_x:x\in X\cap M_\alpha\}$.
But now take $\beta\in A$ above~$\alpha$ and take $x$ and~$y$ 
in~$M_\alpha$ such that $x_\beta\in B_x$ and $y_\beta\in B_y$.
It follows readily that $B_x\subseteq U$ and $B_y\subseteq V$, 
so that $x_\beta$ and~$y_\beta$ do have disjoint neighbourhoods that
belong to~$\calB_\alpha$.

This contradiction concludes the proof.
\end{proof}

\subsection{$\omega_1$-free spaces}

Here we include two technical results on $\omega_1$-free spaces that will be 
useful in Section~\ref{section.forcing}.
The first follows from \cite{MR2586011}*{Lemma~2.2} by setting 
$\rho=\mu=\aleph_1$; we give a proof for completeness and to illustrate
the use of elementary sequences.

\begin{theorem} \label{separable}
If every separable subspace of a compact $\omega_1$-free space is 
first-countable, then the space is first-countable. 
\end{theorem}

\begin{proof} 
Assume that $X$ is compact and $\omega_1$-free. 
Working contrapositively we assume that $X$ is not first-countable and produce
a separable subspace that is not first-countable.

To this end we let $\omegaoneseq{M}$ be an elementary 
sequence for~$X$. 
By elementarity there will be a point~$x$ in $X\cap M_0$ at which $X$~is not
first-countable; by compactness this means that $\{x\}$ is not a 
$G_\delta$-set in~$X$.
This implies that for each $\alpha\in \omega_1$ there is a point 
$x_\alpha\in M_{\alpha+1}$ that is in
$\bigcap\{U:U\in M_\alpha$ and $U$~is open in~$X\}$
and distinct from~$x$.

Since $X$ is $\omega_1$-free, there is a complete accumulation point, $z$,
of $\{x_\alpha:\alpha \in \omega_1\}$ that is distinct from~$x$. 
Since $X$ has countable tightness there is a $\delta\in\omega_1$ such that 
$z$~is in the closure of~$X\cap M_\delta$. 

We show that $\cl(X\cap M_\delta)$ is not first-countable at~$x$.
Indeed, if $W$~is an open neighbourhood of~$x$ that belongs to~$M_{\delta+1}$
then $x_\alpha\in W$ for all $\alpha>\delta$ and in particular $z\in\cl W$.
This more than shows that $M_{\delta+1}$ does not contain a countable family
of neighbourhoods of~$x$ that would determine a local base at~$x$
in $\cl(X\cap M_\delta)$; by elementarity there is no such family at all.
\end{proof}

For the next result we need a piece of notation and the notion of a 
local $\pi$-net.

If $\calF$ is filter on a set $X$ then $\calF^+$ denotes the family of 
sets that are positive with respect to~$\calF$, i.e.,
$G\in\calF^+$ iff $G$~intersects every member of~$\calF$.

A \emph{local $\pi$-net} at a point, $x$, of a topological space is a family, 
$\calA$,
of non-empty subsets of the space such that every neighbourhood of~$x$ contains
a member of~$\calA$.
Clearly $\{x\}$ is a local $\pi$-net at~$x$ but it may not always be 
a very useful one; the next result produces a local $\pi$-net that consists of
somewhat larger sets.

\begin{theorem} \label{filter}
Let $X$ be a compact space of countable tightness and let $\calF$ be a 
countable filter base in~$X$.
Then there is a point $x$ in $\bigcap\{\cl{F}:F\in\calF\}$
that has a countable $\pi$-net that is contained in~$\calF^+$. 
\end{theorem}

\begin{proof} 
Without loss of generality we assume that $\calF$ is enumerated
as $\{F_n:n\in\omega\}$ such that $F_{n+1}\subseteq F_n$ for all~$n$.
Take a sequence $\omegaseq{a}$ in~$X$ such that
$a_n\in F_n$ for all~$n$ and let $K$ be the set of cluster points
of this sequence.

If $K$ has an isolated point, $x$, then some subsequence
of $\omegaseq{a}$ converges to~$x$; the tails of that
sequence form the desired $\pi$-net at~$x$. 

In the other case there is a point $x$ such that $K$ has a countable local
$\pi$-base $\{U_n:n\in\omega\}$ at~$x$.
Shrink each member~$U_n$ of this local $\pi$-base to a compact relative 
$G_\delta$-set~$G_n$.

Write $G_n=K\cap\bigcap_{m\in\omega}O_{n,m}$, where each $O_{n,m}$~is
open in~$X$ and $\cl O_{n,m+1}\subseteq O_{n,m}$ for all~$m$.
Choose an infinite subset~$A_n$ of~$\{a_k:k\in\omega\}$ such that
$A_n\setminus O_{n,m}$~is finite for all~$m$.
Observe that all accumulation points of~$A_n$ belong to~$G_n$.

Now, if $O$~is an open set that contains~$x$ then $G_n\subseteq O$ for some~$n$ 
and hence $A_n\setminus O$~is finite.
This shows that the family $\{A_n\setminus F:n\in\omega$, $F$~is finite$\,\}$ 
is the desired $\pi$-net at~$x$.
\end{proof}

\subsection{A strengthening of the Fr\'echet-Urysohn property}

We shall need a version of the Fr\'echet-Urysohn property where the
convergent sequences are guided by ultrafilters.

\begin{definition} \label{frechet}
A space $X$ will be said to be \emph{ultra-Fr\'echet} if it has countable 
tightness and for each countable subset~$D$ of~$X$ and free 
ultrafilter $\calU$ on~$D$ there is a countable subfamily~$\calU'$ 
of~$\calU$ with the property that every infinite pseudointersection 
of~$\calU'$ converges.
\end{definition}

A set $P$~is an infinite pseudointersection of a family~$\calF$ of subsets
of~$\omega$ if $P\setminus F$ is finite for all~$F\in\calF$.

We shall use this property in the proof of Theorem~\ref{main}, where
we will have to distinguish between a subspace being ultra-Fr\'echet
or not.
To see how the property will be used we prove the following lemma.

\begin{lemma}\label{lemma:use_of_ultra}
Let $D$ be a countable subset of an ultra-Fr\'echet space $X$ and let 
$x\in\cl D$.
Then whenever $\omegaseq{A}$ is a decreasing sequence of 
infinite subsets of~$D$ such that $x\in\bigcap_{n\in\omega}\cl A_n$ there is an 
infinite pseudointersection of the~$A_n$ that converges to~$x$.
\end{lemma}

\begin{proof}
Let $\omegaseq{A}$ be given and let $\calU$ be an ultrafilter
on~$D$ that converges to~$x$ and contains all~$A_n$.
Let $\calU'$ be a countable subset of~$\calU$ as in the definition of
ultra-Fr\'echet.

We first claim that \emph{every} infinite pseudointersection of~$\calU'$ 
converges to~$x$.
Indeed, since the union of two pseudointersections is again a pseudointersection
all pseudointersections converge to the same point, $y$ say.
Next, every neighbourhood of~$x$ contains such a pseudointersection, 
which implies that $y$~belongs to every neighbourhood of~$x$ and therefore
$y=x$.

To finish the proof take any infinite pseudointersection of the countable
family $\calU'\cup\{A_n:n\in\omega\}$.
\end{proof}

\subsection{A preservation result}

We finish this section by quoting a preservation result on the Lindel\"of 
property.
The Tychonoff cube $[0,1]^{\omega_1}$ is compact but when we pass to a forcing
extension the ground model cube is at best a (proper) dense subset
of the cube in the extension and thus no longer a compact space.
Under certain circumstances, however, it will still posses
the Lindel\"of property.

\begin{proposition}[\cite{JunKosz}]\label{kosz}
If\/ $\Poset$ is a poset whose every finite power satisfies the countable
chain condition then in any forcing extension by~$\Poset$ 
the set of points in $[0,1]^{\omega_1}$ from the ground model is still 
Lindel\"of.
\qed
\end{proposition}

Note that this result applies to closed subsets of $[0,1]^{\omega_1}$ as well.

\section{Forcing extensions}
\label{section.forcing}

In this section we prove the main result of this paper; it establishes
first-countabi\-li\-ty of $\omega_1$-free spaces in a variety of models.
The better known of these are the Cohen model, the random real model
and Hechler's models with various cofinal subsets in~$\functions\omega\omega$
with the order~$<^*$.

We begin by defining the particular type of poset our result will apply to.

\subsection{$\omega_1$-finally property K}

We recall that a subset $A$ of a poset $\Poset$ is \emph{linked}
if any two elements are compatible, i.e., whenever $p$ and $q$ are in~$A$ 
there is an $r\in\Poset$ such that $r\le p$ and $r\le q$.
A poset has \emph{property~K}, or \emph{satisfies the Knaster condition}
if every uncountable subset has an uncountable linked subset.
Any measure algebra has property~K and any finite support iteration of
property K posets again has property~K.  

Our result uses a modification of this notion.

\begin{definition}\label{def.finally} 
We say that a poset $\Poset$ is $\omega_1$-finally property K if for each 
completely embedded poset~$\Qoset$ of cardinality at most~$\omega_1$ the
quotient $\Poset/\fname G$ is forced, by~$\Qoset$, to have property~K.
\end{definition}

A poset $\Qoset$ is completely embedded in~$\Poset$ if for every generic 
filter~$H$ on~$\Poset$ the intersection~$H\cap\Qoset$ is generic on~$\Qoset$.
As explained in~\cite{Kunen80}*{Chapter~VII.7} the factor (or quotient) 
$\Poset/\fname G$ is a name for the poset obtained from a generic 
filter~$G$ on~$\Qoset$ in the following way: it is the subset of those 
elements of~$\Poset$ that are compatible with all elements of~$G$.
We shall use the important fact that $\Poset$ is forcing equivalent
to the two-step iteration $\Qoset*(\Poset/\fname G)$,
\cite{Kunen80}*{Chapter~VII, Exercises D3--5}
or \cite{Kunen2011}*{Lemma~V.4.45}.

\subsection{Forcing and elementarity}\label{subsec:forcing-elementarity}

Throughout our proofs we will be working with elementary sequences and we 
shall frequently be using the following fact, a proof of which can be found 
in~\cite{MR1623206}*{Theorem~III\,2.11}.

\begin{proposition}
Let $M\prec H(\theta)$ and let $\Poset\in M$ be a poset.
Then $M[G]$ is an elementary substructure of~$H(\theta)[G]$
(which is the~$H(\theta)$ of~$V[G]$).
\qed   
\end{proposition}

Here $M[G]=\{\val_G(\tau):\tau\in M$ and $\tau$~is a $\Poset$-name$\}$.

The general situation that we will consider is one where we have
an elementary sequence $\omegaoneseq{M}$
and a poset~$\Poset$ that belongs to~$M_0$.
The union $M=\bigcup_{\alpha\in\omega_1}M_\alpha$ is also an elementary 
substructure
of the~$H(\theta)$ under consideration.

Since we will be assuming $\CH$ it follows that 
$\functions\omega M\subseteq M$; using this one readily proves the following
proposition.

\begin{proposition}
If $\Poset\in M$ is a partial order that satisfies the countable chain  
then $\Poset\cap M$ is a complete suborder of~$\Poset$.
\qed
\end{proposition}

Thus the intersection $G_M$ of a generic filter~$G$ on~$\Poset$ 
with~$M$ will be generic on~$\Poset_M$.

Furthermore, if $\fname X\in M_0$ is a $\Poset$-name for a compact space then
the above implies that in~$V[G]$ the sequence 
$\langle M_\alpha[G]:\alpha\in\omega_1\rangle$ is an elementary sequence 
for~$X$.

Finally, a $\Poset$-name $\fname A$ for a subset of~$\omega$ can
be represented by the subset $\{\orpr pn:p\forces n\in\fname A\}$
of $\Poset\times\omega$ and even by a countable subset of this product:
simply choose, for each~$n$, a maximal antichain~$A_n$ in 
$\{p:p\forces n\in\fname A\}$ and let $A'=\{\orpr pn:p\in A_n\}$.
Now if $\fname A\in M$ then, by elementarity there is such a countable~$A'$ 
in~$M$, and then $A'\in M_\alpha$ for some~$\alpha$.
But then $A'\subseteq M_\alpha$ and therefore 
$\val_G(\fname A)\in M_\alpha[G_M]\subseteq M[G_M]$. 

The facts above are well known but we reviewed them because they are crucial 
to some of our arguments.

\subsection{Pre-Luzin gaps}

In our proof we shall construct converging $\omega_1$-sequences in an 
intermediate model and use the following combinatorial structure to lift these
to the full generic extension.

\begin{definition} 
For a countable set $D$, a family 
$\{\orpr{a_\alpha}{b_\alpha}:\alpha\in\omega_1\}$ of ordered pairs of disjoint
subsets of~$D$ will be called a pre-Luzin gap if for all $E\subset D$ the set 
of~$\alpha$ such that $E\cap (a_\alpha\cup b_\alpha) =^* a_\alpha$ is countable.
\end{definition}

The lifting of the sequence to the full extension is accomplished using
the following Lemma.

\begin{lemma}\label{preluzin} 
If\/ $\Poset$ has property K and $\{\orpr{a_\alpha}{b_\alpha}:\alpha\in\omega_1\}$
is a pre-Luzin gap on~$\omega$ then it remains a pre-Luzin gap in every 
forcing extension by $\Poset$. 
\end{lemma}

\begin{proof} 
Let $\fname E$ be a $\Poset$-name of a subset of~$\omega$. 
Arguing contrapositively we assume there is an uncountable subset~$A$ 
of~$\omega_1$ so that for each $\alpha\in A$ there are a 
$p_\alpha\in \Poset$ and an integer $n_\alpha$ such that 
$p_\alpha\forces_{\Poset}\fname E\cap(a_\alpha\cup b_\alpha)\setminus n_\alpha
     =a_\alpha\setminus n_\alpha$. 
We shall build a subset~$E$ of~$\omega$ such that 
$E\cap (a_\alpha\cup b_\alpha) =^* a_\alpha$ for uncountably 
many~$\alpha \in A$.

We apply property~K and assume, without loss of generality,
that $\{ p_\alpha : \alpha \in A\}$ is linked and that there is a single 
integer~$n$ such that $n_\alpha = n$ for all $\alpha\in A$.

We let $E=\bigcup\{a_\alpha\setminus n:\alpha\in A\}$.
To verify that $E$~is as required we let $\alpha\in A$ and $j\in E\setminus n$.
Fix $\beta\in A$ such that $j\in a_\beta$.
Then $p_\beta\forces j\in\fname E$ and 
$p_\alpha\forces\fname E\cap b_\alpha\subseteq n$; 
as $p_\alpha$ and $p_\beta$ are compatible this implies that $j\notin b_\alpha$.
\end{proof}

\subsection{The main theorem}

\begin{theorem}[CH]\label{main}
Let $\Poset$ be a poset that is $\omega_1$-finally property K.
Then in every generic extension by $\Poset$ every compact $\omega_1$-free
space is first-countable and L-reflecting.
\end{theorem}

We prove this theorem in three steps.

Let $\fname X$ be a $\Poset$-name for a compact Hausdorff space that is
$\omega_1$-free.
We take an elementary sequence $\omegaoneseq{M}$ such
that $\Poset$ and $\fname X$ are in~$M_0$. 

We continue to write $M=\bigcup_{\alpha\in\omega_1}M_\alpha$,
$\Poset_M=\Poset\cap M$
and $G_M=G\cap M$.
To save on writing we put $N_\alpha=M_\alpha[G]$ and $N=M[G]$.

\begin{proposition}\label{prop:ultra-Frechet}
In $V[G]$ the space $X$ is ultra-Fr\'echet.  
\end{proposition}

\begin{proof}
We assume $X$ is not ultra-Fr\'echet and take a countable set $D$ witnessing 
this --- we know already that $X$~has countable tightness.
Let $e:\omega\to D$ be a bijection and let $\calU$ be an ultrafilter 
on~$\omega$ such that every countable subfamily of it has an infinite 
pseudointersection whose image under~$e$ does not converge.
Let $z$ be the limit of~$e(\calU)$ and observe that 
$\{z\}\neq\bigcap\{\cl e[U]:U\in\calU'\}$ whenever $\calU'$ is a countable
subfamily of~$\calU$; indeed, if equality were to hold then the image of every 
infinite pseudointersection of~$\calU'$ would converge to~$z$.

By elementarity we can assume that $D$, $e$, $\calU$ and (hence) $z$ belong 
to~$N_0$.

For each $\alpha$ we find a point $x_\alpha\in X$ and a countable 
family~$\calA_\alpha$ of subsets of~$\omega$ as follows.
We apply the property above to $\calU\cap N_\alpha$ and fix $y\in N_{\alpha+1}$
such that $y\neq z$ and $y\in\bigcap\{\cl e[U]:U\in\calU\cap N_\alpha\}$.
Next we take a neighbourhood~$W$ of~$y$ with $z\notin\cl W$.

We apply Theorem~\ref{filter} to the filterbase generated by
$\{W\cap D\}\cup(e(\calU)\cap N_\alpha)$ to 
find a point~$x_\alpha$ and a countable local $\pi$-net at~$x_\alpha$ that can be 
written as~$\{e[A]:A\in\calA_\alpha\}$, where $\calA_\alpha$~is a countable 
subfamily of~$\bigl(\{e\preim[W]\}\cup(\calU\cap N_\alpha)\bigr)^+$, 
which itself is a subfamily of~$(\calU\cap N_\alpha)^+$.
Note that $x_\alpha\neq z$ because $x_\alpha\in\cl W$.

By elementarity the choices above --- $W$, $x_\alpha$ and $\calA_\alpha$ ---
can all be made in~$N_{\alpha+1}$.

As noted above in Subsection~\ref{subsec:forcing-elementarity}, 
since $\calA_\alpha$ is a \emph{countable} family of subsets
of~$\omega$ we can find a name for it that is actually a subset 
of~$M_{\alpha+1}$.
Therefore we know that each $\calA_\alpha$ and its members belong 
to~$V[G_M]$.

Since we assume that $X$~is $\omega_1$-free the sequence
$\omegaoneseq{x}$ will have (at least) two distinct
complete cluster points, so that there are two open sets~$O_1$ and~$O_2$
with disjoint closures that each contains $x_\alpha$ for uncountably 
many~$\alpha$.

Now we consider $V[G]$ as an extension of~$V[G_M]$ by the poset~$\Poset/G_M$ 
and choose a sequence $\omegaoneseq{p}$ of conditions
in the latter,
two strictly increasing sequences of ordinals
$\langle\beta_\alpha:\alpha\in\omega_1\rangle$ and 
$\langle\gamma_\alpha:\alpha\in\omega_1\rangle$, and
two sequences
$\omegaoneseq{a}$ and 
$\omegaoneseq{b}$ of subsets of~$\omega$, 
such that
$$
p_\alpha\forces \fname x_{\beta_\alpha}\in\fname O_1
\text { and }
p_\alpha\forces \fname x_{\gamma_\alpha}\in\fname O_2
$$
and $a_\alpha\in\calA_{\beta_\alpha}$ and $b_\alpha\in\calA_{\gamma_\alpha}$, and
$$
p_\alpha\forces e[a_\alpha]\subseteq\fname O_1
\text { and }
p_\alpha\forces e[b_\alpha]\subseteq\fname O_2
$$
We apply the ccc to find $q\in\Poset/G_M$ that forces
$G\cap \{p_\alpha:\alpha\in\omega_1\}$ to be uncountable.

In $V[G]$ we form $A=\{\alpha:p_\alpha\in G\}$ and $E=e\preim[O_1]$;
then $a_\alpha\subseteq E$ and $b_\alpha\cap E=\emptyset$ for $\alpha\in A$,
so that $\{\orpr{a_\alpha}{b_\alpha}:\alpha\in\omega_1\}$ is not a pre-Luzin
gap in $V[G]$.

However, in $V[G_M]$ the set $\{\orpr{a_\alpha}{b_\alpha}:\alpha\in\omega_1\}$
\emph{is} a pre-Luzin gap.
For if $E\subseteq\omega$ belongs to $V[G_M]$ then it belongs to $N_\alpha$
for some~$\alpha$ and either it or its complement belongs to~$\calU$,
say $E\in\calU$.
But then $a_\beta$ and $b_\beta$ both meet $E$~in an infinite set 
whenever $\beta>\alpha$.
\end{proof}

The next step is to prove that $X$~is first-countable.

\begin{proposition}\label{prop:1st-ctble}
In $V[G]$ the space $X$ is first-countable.  
\end{proposition}

\begin{proof}
We assume it is not and apply Theorem~\ref{separable} to find a countable
subset~$D$ of~$X$ and a point~$z$ in~$\cl D$ that does not have a countable
local base in~$\cl D$. 
By elementarity $z$ and $D$ can be found in~$N_0$; as above we take a 
bijection~$e$ from~$\omega$ to~$D$, also in~$N_0$.

Using the fact that $X$ is ultra-Fr\'echet we shall construct an 
$\omega_1$-sequence that converges to~$z$.
To this end we observe that for every $\alpha$ one can build a family 
$\{A(\alpha,s):s\in\functions{<\omega}2\}$
of subsets of~$\omega$ that has the following 
properties
\begin{enumerate}
\item $A(\alpha,\emptyset)=\omega$,
\item $A(\alpha,s)=A(\alpha,s*0)\cup A(\alpha,s*1)$ 
      (the $*$ denotes concatenation),
\item $A(\alpha,s*0)\cap A(\alpha,s*1)=\emptyset$,
\item for every subset $A$ of~$\omega$ that is in~$N_\alpha$ there is 
      an~$n$ such that the family $\{A(\alpha,s):s\in\functions n2\}$ refines
       $\{A, \omega\setminus A\}$.
\end{enumerate}
This can be done in a simple recursion, using the fact that $N_\alpha$~is
countable; by elementarity we can assume that the 
family $\{A(\alpha,s):s\in\functions{<\omega}2\}$ belongs to~$N_{\alpha+1}$.
Each $A(\alpha,s)$ determines, via~$e$, a subset $D(\alpha,s)$ of~$D$.

Fix an $\alpha$ and form 
$$
C_\alpha=\{s\in\functions{<\omega}2: z\notin\cl D(\alpha, s)\};
$$
the set $\bigcap \{\cl D\setminus\cl D(\alpha, s)\}$ is a $G_\delta$-set
in~$\cl D$ and hence it is not equal to~$\{z\}$.
Thus we may pick $y_\alpha\neq z$ in this intersection and a function~$f_\alpha$
in~$\functions\omega2$ such that $y_\alpha\in\cl D(\alpha,f_\alpha\restr m)$ 
for all~$m$; note that then also $z\in\cl D(\alpha,f_\alpha\restr m)$ 
for all~$m$.

Now apply Lemma~\ref{lemma:use_of_ultra} and choose infinite 
pseudointersections $a_\alpha$ and $b_\alpha$ of the 
$A(\alpha,f_\alpha\restr m)$ such that $e[a_\alpha]$ and~$e[b_\alpha]$
converge to~$z$ and~$y_\alpha$ respectively.
By elementarity $C_\alpha$, $y_\alpha$, $f_\alpha$, $a_\alpha$ and~$b_\alpha$
can all be chosen in~$N_{\alpha+1}$, and as in the previous proof,
the countable objects --- in particular $a_\alpha$ and~$b_\alpha$ ---
actually belong to $M_{\alpha+1}[G_M]$.

Again as in the previous proof: any subset of~$\omega$ that belongs 
to~$V[G_M]$ has a name that is in~$M$ and hence belongs to $N_\alpha$ for 
some~$\alpha$, which implies that it either contains or is disjoint from 
both $a_\beta$ and~$b_\beta$ whenever $\beta\ge\alpha$.
Thus, $\{\orpr{a_\alpha}{b_\alpha}:\alpha\in\omega_1\}$ is a pre-Luzin gap
in $V[G_M]$ and hence, because $\Poset/G_M$ has property~K,  
it is also a pre-Luzin gap in~$V[G]$.

We show that this implies that $\omegaoneseq{y}$ converges to~$z$.
Indeed, let $U$ be a neighbourhood of~$z$ and consider $e\preim[U\cap D]$;
this set contains a cofinite part of every set~$a_\alpha$ and hence
an infinite part of all but countably many~$b_\alpha$.
This implies that $y_\alpha\in\cl U$ for all but countably many~$\alpha$. 
\end{proof}

\subsection*{Proof that $X$ is L-reflecting}

Since $\omegaoneseq{M}$ is an arbitrary elementary sequence
it suffices to show that $X\cap N$~is Lindel\"of.
This will require some more notation.

To begin we fix, in~$V$, a cardinal~$\kappa$ and we assume that $\fname X$~is 
forced to be a subset of $[0,1]^\kappa$; we can take $\kappa$ in~$M_0$.
We also write $\Gamma=\kappa\cap M$ and let $\pi_\Gamma$ denote
the projection of~$[0,1]^\kappa$ onto~$[0,1]^\Gamma$.
We shall show three things
\begin{enumerate}
\item $\pi_\Gamma$ is a homeomorphism between $X\cap N$
       and $\pi_\Gamma[X\cap N]$ in~$V[G]$, 
\item $\pi_\Gamma[X\cap N]$ is in~$V[G_M]$, and
\item $\pi_\Gamma[X\cap N]$ is a closed subset of $[0,1]^\Gamma$ in~$V[G_M]$.
\end{enumerate}

Proposition~\ref{kosz} then implies that $\pi_\Gamma[X\cap N]$ is 
Lindel\"of in~$V[G]$ and hence that $X\cap N$~is Lindel\"of too.

The first item is a consequence of the first-countability of~$X$.

\begin{lemma}
The map $\pi_\Gamma$ is a homeomorphism between $X\cap N$ 
and $\pi_\Gamma[X\cap N]$.  
\end{lemma}

\begin{proof}
Because $X$~is first-countable there is a countable local base at each
point of~$X\cap N$ that consists of basic open sets and that, by elementarity,
may be taken to be a member of~$N$.
The latter means that all members of such a local base have their
supports in~$\Gamma$.
This is enough to establish the lemma.
\end{proof}

In the proof of the other two statements we abbreviate $M_\alpha[G_M]$ 
by~$M_\alpha^+$ and $M[G_M]$ by~$M^+$.
We also need a way to code members of~$[0,1]^\kappa$ that makes it easy
to calculate (names for) projections of members of~$X$.
A point of~$[0,1]^\Gamma$ is determined by a function 
$x:\kappa\times\omega\to2$: its $\gamma$th coordinate is given 
by~$\sum_{n\in\omega}x(\gamma,n)\cdot2^{-n-1}$.

If $\fname x\in M$ is such a name and $x=\val_G(\fname x)$ then one readily
checks that $\pi_\Gamma(x)=\val_{G_M}(\fname x)$.

We let $X^+$ denote the set of $\Poset$-names of such functions that
are forced by $\Poset$ to determine members of~$X$.
Note that $X^+\in M_0$ by elementarity.

Using these names it is easy to prove the second item in our list.

\begin{lemma}
The set $\pi_\Gamma[X\cap N]$ belongs to $V[G_M]$.  
\end{lemma}

\begin{proof}
Using the coding described above it follows that
$\pi_\Gamma[X\cap N]=\{\val_{G_M}(\fname x):x\in X^+\}$;
the latter set belongs to $V[G_M]$.
\end{proof}

In preparation for the proof of the third item in our list we prove.

\begin{lemma}\label{lemma:equal-alpha}
For every $\alpha$ we have 
$\pi_\Gamma[X\cap N_\alpha]=\pi_\Gamma[X]\cap M_\alpha^+$.  
\end{lemma}

\begin{proof}
The equality $\val_{G_M}(\fname x)=\pi_\Gamma\bigl(\val_G(\fname x)\bigr)$
for $\fname x\in M_\alpha$ establishes the inclusion
$\pi_\Gamma[X\cap N_\alpha]\subseteq \pi_\Gamma[X]\cap M_\alpha^+$.

For the converse let $\fname x\in M_\alpha$ be such that 
$x_M=\val_{G_M}(\fname x)\in\pi_\Gamma[X]$; then $x=\val_G(\fname x)$
belongs to~$N_\alpha$ and, by elementarity, $y=x\restr(\kappa\times\omega)$
is a function that also belongs to~$N_\alpha$ and whose domain contains
$\Gamma\times\omega$.
By elementarity $\dom y=\kappa\times\omega$ and so $x_M=\pi_\Gamma(y)$.  
\end{proof}

The proof of our third statement will almost be a copy of that of 
Proposition~\ref{prop:ultra-Frechet}. 

\begin{proposition}\label{prop:equal}
In $V[G_M]$ the set $\pi_\Gamma[X\cap N]$ is closed in~$[0,1]^\Gamma$.  
\end{proposition}

\begin{proof}
In $V[G_M]$ let $z\in\cl\pi_\Gamma[X\cap N]$.
Of course $z$~is a point of~$\cl\pi_\Gamma[X\cap N]$ as computed in~$V[G]$ 
as well and hence $z\in\pi_\Gamma[X]$ as the latter set is compact; 
the task is to find $x\in X\cap N$ such that $z=\pi_\Gamma(x)$.

In $V[G]$ the set $\pi_\Gamma[X]$ is of countable tightness.
Hence $z$~is in the closure of $\pi_\Gamma[X\cap N_\delta]$ for 
some $\delta<\omega_1$.

We assume, to reduce indexing, that $\delta=0$ and we write $D$ 
for $\pi_\Gamma[X\cap N_0]$; we also take an enumeration $e:\omega\to D$  
that belongs to~$M_1^+$; we shall use $e$ also, implicitly, 
to enumerate~$X\cap N_0$.

We take an ultrafilter~$\calU$ on~$\omega$ such that $e(\calU)$~converges 
to~$z$.
Note that, in contrast with the proof of Proposition~\ref{prop:ultra-Frechet},
neither $z$ nor~$\calU$ need belong to~$M^+$.
However, because $\functions\omega M\subseteq M$ and because $\Poset_M$~is a 
ccc poset of cardinality~$\aleph_1$ we also have 
$\functions\omega{M^+}\subseteq M^+$.
Therefore we know that for every~$\alpha$ there is $\beta_\alpha>\alpha$ 
such that $\calU_\alpha=\calU\cap M_\alpha^+\in M_{\beta_\alpha}^+$.

Thus, if there is some $\alpha$ such that 
$\{z\}=\bigcap\{\cl e[U]:U\in\calU_\alpha\}$
then $z\in M_{\beta_\alpha}^+$ and Lemma~\ref{lemma:equal-alpha} applies
to show that $z\in\pi_\Gamma[X\cap N_{\beta_\alpha}]$.
From now on we assume 
$\{z\}\neq\bigcap\{\cl e[U]:U\in\calU_\alpha\}$ for all~$\alpha$
and follow the proof Proposition~\ref{prop:ultra-Frechet}
to reach a contradiction.

The only modification that needs to be made is when choosing the 
point~$x_\alpha$ and the family~$\calA_\alpha$.
Our assumption now yields that $\bigcap\{\cl e[U]:U\in\calU_\alpha\}$ has 
more than one point, hence there are two basic open sets $W_1$ and~$W_2$
in~$M_{\beta_\alpha}$ with disjoint closures that both meet this intersection.
We let $W$ be one of the two that does not have $z$ in its closure.

We now use first-countability of~$X$ to find $x_\alpha\in X\cap N_{\beta_\alpha}$
and an infinite pseudointersection~$c_\alpha$ 
of~$\calU_\alpha\cup\{e\preim[W]\}$, also in~$N_{\beta_\alpha}$, such
that $e[c_\alpha]$ converges to~$x_\alpha$ 
(remember that $e$ also enumerates $X\cap N_0$);
$\calA_\alpha$~consists of the cofinite subsets of~$c_\alpha$.

From here on the proof is the same as that of 
Proposition~\ref{prop:ultra-Frechet}.
\end{proof}

\section{Examples}

\subsection*{Juh\'asz' question}

Since our main result establishes a consistent positive answer to Juh\'asz' 
question we should begin by recording a consistent negative answer as well.

\begin{example}[\cite{JuhKoszSouk}] 
It is consistent to have a compact $\omega_1$-free space that
is not first-countable.
\end{example}

The space is the one-point compactification of a locally compact, 
first-countable and initially $\omega_1$-compact space that is locally of 
cardinality~$\aleph_1$. 
The space is not L-reflecting either but this is not easily shown so we omit 
the proof.

\begin{question}\label{q2}  
If $X$ is locally compact, not Lindel\"of, and initially $\omega_1$-compact,
does it fail to be L-reflecting?
\end{question}

\subsection*{Hu\v{s}ek versus Efimov}
We can also use our main result to show that Hu\v{s}ek's question is 
strictly weaker than Efimov's:
in~\cite{DwSh:984} it is shown that $\bee=\cee$ implies there is an Efimov 
space, that is, a compact Hausdorff space that contains neither a 
converging $\omega$-sequence nor a copy of~$\beta\N$.
Since we can use Hechler forcing to create models for~$\bee=\cee$, 
where $\cee$~can have any regular value we please, we get a slew of models 
where Hu\v{s}ek's question has a positive answer and Efimov's a negative~one.

\subsection*{The need for property K}

To demonstrate the need for property~$K$ in the proof of Theorem~\ref{main} 
we quote the following example. 

\begin{example}[\cite{Kosz}]\label{ex1}
There is a ccc poset $\Poset$, with a finite powers also ccc,
that forces the existence of a compact first-countable space that
is not weakly L-reflecting.
\end{example}

This space is constructed in Theorems~7.5 and~7.6 of~\cite{Kosz}. 
Theorem~7.5 produces a compact space~$K$ with $\functions{\omega_1}2$
as its underlying set and the property that whenever a point~$f$ and a sequence 
$\omegaoneseq{f}$ in~$\functions{\omega_1}2$ are given
such that $f_\alpha \cap f \in 2^\alpha$ for all~$\alpha$
then in~$K$ the point~$f$ is the limit of the converging $\omega_1$-sequence 
$\omegaoneseq{f}$. 
Theorem~7.6 then produces a compact first-countable space $X$ that maps 
onto~$K$. 

The poset is constructed in a ground model~$V$ that satisfies~$\CH$. 
We let $T=(\functions{<\omega_1}2)^V$ and we choose for each $t\in T$ 
an $f_t\in \functions{\omega_1}2$ such that $t\subset f_t$.
The closure of the set $Y=\{f_t:t\in T\}$ will contain all the cofinal 
branches of~$T$, and so will contain $\aleph_2$~many converging 
$\omega_1$-sequences with distinct limits of~$K$.
It then follows that any subset of~$X$ that maps onto~$Y$ will not be contained 
in a Lindel\"of subset of cardinality~$\aleph_1$.

\begin{question} 
Are compact spaces with small diagonal 
metrizable in the model described in Example~\ref{ex1}?
\end{question}

While we do not know the answer to this question, let us remark that the space 
constructed in Example~\ref{ex1} does not have a small diagonal. 
In fact it was the space~$X$ of Example~\ref{ex1} that was the motivation 
for Proposition~2.4 of~\cite{DH1}. 
The space~$X$ has copies of Cantor sets and $\omega_1$-sequences that 
co-countably converge to these. 
By the aforementioned proposition this implies that $X$ does not have a 
small diagonal.

\subsection*{Acknowledgement}
We thank the referee for suggesting that the proof of the L-reflecting 
property of~$X$ could be simplified and made more direct.

\end{document}